\documentclass[11pt]{amsart}

\usepackage{hyperref}
\usepackage{amssymb, amsfonts, amsmath, amsthm, enumerate}
\usepackage[margin=1in]{geometry}

\newcommand{\Z}{\mathbf{Z}}
\newcommand{\R}{\mathbf{R}}
\newcommand{\Q}{\mathbf{Q}}
\renewcommand{\phi}{\varphi}

\DeclareMathOperator*{\conv}{conv}
\DeclareMathOperator*{\cone}{cone}

\renewcommand{\P}{{\mathcal P}}
\newcommand{\e}{{\mathbf e}}
\newcommand{\m}{{\mathbf m}}
\renewcommand{\v}{{\mathbf v}}

\newcommand{\x}{{\mathbf x}}
\newcommand{\y}{{\mathbf y}}
\newcommand{\z}{{\mathbf z}}
\def\th{^{\text{th}}}
\newcommand{\vol}{\operatorname{vol}}
\newcommand{\F}{\mathcal{F}}
\newcommand{\Solid}{\operatorname{Solid}}

\newcommand{\eps}{\varepsilon}

\newtheorem{theorem}{Theorem}

\newtheorem{proposition}[theorem]{Proposition}
\newtheorem{lemma}[theorem]{Lemma}

\newtheorem{eg}[theorem]{Example}

\newtheorem*{definition}{Definition}

\newcommand\comment[1]{}                
\renewcommand\comment[1]{\emph{[#1]}}           

\title{Positivity Theorems for Solid-Angle Polynomials}

\author{Matthias Beck}
\address{Department of Mathematics,
San Francisco State University,
San Francisco, CA 94132, U.S.A.}
\email{beck@math.sfsu.edu}

\author{Sinai Robins}
\address{Division of Mathematical Sciences,
School of Physical and Mathematical Sciences,
Nanyang Technological University,
Singapore, 637371}
\email{rsinai@ntu.edu.sg}

\author{Steven V Sam}
\address{Department of Mathematics,
Massachusetts Institute of Technology,
Cambridge, MA 02139, U.S.A.}
\email{ssam@math.mit.edu}

\date{June 26, 2015}

\thanks{The authors thank an anonymous referee for helpful comments.
  Matthias Beck was partially supported by NSF grant DMS-0810105, 
  Sinai Robins was supported by an NTU academic research fund SUG grant, 
  and Steven Sam was partially supported by NDSEG and NSF graduate fellowships.
  }



\begin{document}

\maketitle

\begin{abstract} For a lattice polytope $\P$, define $A_{\P}(t)$ as
  the sum of the solid angles of all the integer points in the dilate
  $t\P$. Ehrhart and Macdonald proved that $A_{\P}(t)$ is a polynomial
  in the positive integer variable $t$. We study the numerator
  polynomial of the solid-angle series $\sum_{ t \ge 0 } A_\P(t) z^t$. In
  particular, we examine nonnegativity of its coefficients,
  monotonicity and unimodality questions, and study extremal behavior
  of the sum of solid angles at vertices of simplices. Some of our
  results extend to more general valuations.
\end{abstract}


\section{Introduction.}

Suppose $\P \subset \R^d$ is a $d$-dimensional polytope with integer
vertices (a \emph{lattice polytope}). Unless otherwise stated, we
shall assume throughout that $\P$ is full-dimensional. Let $B(r,\x)$
be the $d$-dimensional ball of radius $r$ centered at the point $\x
\in \R^d$. Then we define the \emph{solid angle at $\x$ with respect
  to} $\P$ to be
\[ 
\omega_{\P}(\x) := \lim_{r \to 0} \frac{\vol (B(r,\x) \cap
  \P)}{\vol(B(r,\x))}\,.
\] 
The fraction above measures the proportion of a small sphere of radius $r$ that intersects the polytope $\P$ and is hence constant for all sufficiently small $r > 0$, so that the limit always exists.   
We note that the notion of a solid angle $\omega_{\P}(\x)$  is equivalent to the notion of the volume of a spherical polytope on the unit sphere, normalized by dividing by the volume of the boundary of the unit sphere. 

We are interested in weighing every lattice point $x \in \Z^d$ by its corresponding solid angle  $ \omega_{t\P}(\x)$, and summing these weights over the whole lattice.  To this end, we consider the function $A_\P \colon \Z_{ >0 } \to \R$ defined by
\[ 
A_{\P}(t) := \sum_{\x \in \Z^d} \omega_{t\P}(\x) \, ,
\] 
which we call the \emph{solid-angle polynomial} of $\P$.  Here $t\P =
\{ t\x \mid \x \in \P \}$ denotes the $t\th$ dilate of $\P$.  The fact
that $A_\P(t)$ is indeed a polynomial follows from Ehrhart's
celebrated theorem that the lattice-point enumerator
\[ 
L_{\P}(t) := \# \left( t\P \cap \Z^d \right) ,
\]
also initially defined only on the positive integers, is in fact a polynomial
in $t \in \Z$.  We will review a few basic facts about $A_\P(t)$ and
$L_\P(t)$ in Section~\ref{backgroundsection}.

Solid-angle polynomials are not as widely known and studied as they
deserve to be. This paper contains a few results that seem basic and
yet have been unknown thus far.     

Some of our results answer various open problems in \cite[Chapter
11]{ccd}. While Ehrhart polynomials can be computed using programs
such as {\tt LattE} \cite{lattemanual,koeppe}, {\tt normaliz}
\cite{normaliz}, or {\tt polylib} \cite{polylib}, there is currently
no software available for computing solid-angle polynomials, so it is
difficult to obtain data for making conjectures. Recent activity on
solid angles can be found in \cite{anglesums} and
\cite{realpolytopes}.

In Section~\ref{formulasection}, we give some formulas related to
calculations of solid angles. We also
address the question of whether there are polytopes for which the
polynomial $A_\P$ has negative coefficients. The equivalent question
regarding Ehrhart polynomials $L_\P(t)$ has a positive answer, as exemplified
by \emph{Reeve's tetrahedron} $\P_h$ whose vertices are $(0,0,0)$,
$(1,0,0)$, $(0,1,0)$, and $(1,1,h)$, with $h$ a fixed positive
integer. This tetrahedron was used by Reeve \cite{reeve} to show that
no linear analogue of Pick's theorem can hold in dimension 3. The Ehrhart
polynomial of $\P_h$ is
\begin{equation} \label{tetrahedron} L_{\P_h}(t) = \frac{h}{6}t^3 +
  t^2 + \left(2-\frac{h}{6}\right)t + 1\, .
\end{equation}
Thus for $h>12$ one obtains Ehrhart polynomials with negative
coefficients. Reeve's tetrahedron allows us to construct solid-angle
polynomials with negative coefficients as well:

\begin{proposition} \label{reeveprop} The linear coefficient of
  $A_{\P_h}(t)$ is negative.
\end{proposition}

In Section~\ref{vertexsection}, we answer an open question (in the negative) raised in
\cite[Chapter 11]{ccd}, namely, whether the solid-angle vertex sum
$\sum_{{\bf v} \text{ a vertex}} \omega_{\Delta}({\bf v})$ is
minimized when $\Delta$ is the regular $d$-simplex. We construct two
infinite families of simplices in dimensions $\ge 3$ that exhibit
extreme asymptotic behavior (approaching 0 and 1/2, respectively) with
respect to the sum of the solid angles at their vertices.

Although the existence of these constructions are special cases of
theorems from \cite{Barnette}, \cite{PerlesShepard1}, and
\cite{PerlesShepard2}, the explicit nature of the examples we give
here complement their existence proofs.

In Section \ref{nonnegativesection} we prove a nonnegativity result.
We define the generating function of $A_\P(t)$ as
\[
\Solid_{\P}(z) := \sum_{t \ge 0} A_{\P}(t) \, z^t .
\]
The fact that $A_\P(t)$ is a degree $d$ polynomial in $t$ is
equivalent to the fact that we can write $\Solid_\P(z)$ as a rational
function in $z$ with denominator $(1-z)^{ d+1 }$. When written this
way, Macdonald \cite{macdonald} proved that the numerator of
$\Solid_\P(z)$ is a palindromic polynomial or, equivalently, that
\begin{equation}\label{macreciprocity}
  A_\P(-t) = (-1)^d A_\P(t) \, .
\end{equation}
We will prove the following.

\begin{theorem}\label{nonnegativethm}
  Given a lattice polytope $\P \subset \R^d$, write
  \[
  \Solid_{\P}(z) = \frac{a_dz^d + a_{d-1}z^{d-1} + \cdots +
    a_0}{(1-z)^{d+1}} \, .
  \]
  Then $a_j > 0$ for $j = 1,2,\dots,d$ and $a_0 = 0$. 
\end{theorem}

If we replace $A_\P(t)$ with $L_\P(t)$, there is a well-known result
due to Stanley \cite{stanleynonnegativity} that $a_j \ge 0$ for
$j=0,1,\dots,d$. 

In an earlier version of this article, we stated two theorems (Theorem 3 and Theorem 4) which Katharina Jochemko and Raman Sanyal have pointed out are not correct as stated. See Section~\ref{sec:valuation} for further discussion, but we explain the setup here for consistency with the previous version of the article.

Let $M$ denote the set of measurable sets in $\R^d$. A {\it valuation} is a function $\nu \colon
M \times \R^d \to \R$ that satisfies inclusion-exclusion:
\[
\nu(K_1 \cup K_2, \x) = \nu(K_1, \x) + \nu(K_2, \x) - \nu(K_1 \cap K_2, \x) \, .
\]
For our purposes, we will replace $M$ by the set of all polyhedral
complexes. A valuation whose codomain is $\R_{ \ge 0 }$ is called
\emph{nonnegative}. A valuation $\nu$ is \emph{translation-invariant}
if $\nu(K+\y,\x+\y) = \nu(K,\x)$ for all $K \in M$, and all $\x, \y
\in \R^d$. See \cite{geometricprob} for more about valuations.

It is clear that $\nu(K,\x) = \omega_K(\x)$ is a
(translation-invariant nonnegative) valuation.

Let $N_\P \colon \Z_{ >0 } \to \R$ be defined through
\begin{equation}\label{Ldef}
N_{\P}(t) := \sum_{\x \in \Z^d} \nu(t\P,\x)
\end{equation}
(in fact, we could replace $\Z^d$ here by an arbitrary lattice),
and let $G_\P(z)$ be the generating function of $N_\P(t)$:
\begin{equation}\label{Gdef}
G_\P(z) := \sum_{ t \ge 0 } N_\P(t) \, z^t .
\end{equation}
McMullen \cite{mcmullenreciprocity} proved that $N_\P(t)$ is a
polynomial if $\P$ is a lattice polytope and, equivalently, that we
can write $G_\P(z)$ as a rational function in $z$ with denominator
$(1-z)^{ d+1 }$. Recall that our polytopes are assumed to be
full-dimensional.

\setcounter{theorem}{4}

In Section~\ref{periodsection}, we discuss a phenomenon that can be
observed with certain \emph{rational polytopes}, i.e., polytopes whose
vertices are in $\Q^d$. In this case, the functions $A_\P(t)$ and
$L_\P(t)$ are examples of a \emph{quasipolynomial}, that is, a
function of the form $c_d(t) \, t^d + c_{ d-1 } (t) \, t^{ d-1 } +
\dots + c_0(t)$, where $c_0(t), \dots, c_d(t)$ are periodic functions
in $t$. The least common multiple of the denominators of the vertex
coordinates of $\P$ is always a period of the coefficient functions of
$A_\P(t)$ and $L_\P(t)$. The recent literature
\cite{deloeramcallister, mcallisterwoods} includes examples of
rational polytopes whose Ehrhart quasipolynomials exhibit \emph{period
  collapse}; that is, they are polynomials. We give a family of
polytopes for which period collapse happens for the solid-angle
quasipolynomials.


\section{Some background.} \label{backgroundsection}

We give a brief review of Ehrhart theory and the theory of solid
angles without any proofs. The interested reader can find proofs and
much more in \cite{ccd}. See also \cite{sam} for proofs with a more
valuation oriented mindset.

As mentioned in the introduction, for a given rational polytope $\P
\subset \R^d$ (for this paragraph, we do not require that $\P$ is
$d$-dimensional), the counting function $L_{\P}(t) := \# \left( t\P
  \cap \Z^d \right)$ is a quasipolynomial in the integer variable $t$.
If $\P$ has integral vertices, then $L_{\P}(t)$ is a polynomial.
Denote the interior of $\P$ as $\P^\circ$ (here we mean topological
interior relative to the affine span of $\P$); the following
\emph{reciprocity law} holds:
\begin{align} \label{reciprocity} 
L_{\P}(-t) = (-1)^{\dim \P} L_{\P^\circ}(t)\,. 
\end{align} 
We will need to know a few properties of Ehrhart polynomials. Namely,
the degree of the polynomial is the dimension of the polytope, and the
leading coefficient is its volume. (We always measure volume of a
polytope relative to its affine span, normalized with respect to the
lattice induced on this affine span.) The second leading coefficient
is half of the sum of the volumes of the facets (the codimension-1
faces). In particular, these two coefficients are always positive.

Now we can explain why $A_\P(t)$ is also a quasipolynomial. For a
face $\F \subseteq \P$, define the solid angle of $\F$ to be the
solid angle of any point in $\F^\circ$, denoted by
$\omega_{\P}(\F)$. Thus
\begin{align} \label{polysum} 
  A_{\P}(t) = \sum_{\F \subseteq \P} \omega_{\P}(\F) L_{\F^\circ}(t)\,,
\end{align}
where the sum is over all faces $\F$ of $\P$ (see also \cite{mcmullenvaluations} for more on the relationship between $A_\P(t)$ and $L_\P(t)$). 
So $A_\P(t)$ is indeed a
quasipolynomial, since $L_{\F^\circ}(t)$ is a quasipolynomial for all
faces $\F$. In particular, if $\P$ is a lattice polytope, then
$A_\P(t)$ is a polynomial in $t$.
 
By \eqref{macreciprocity}, we have $A_\P(0) = 0$. An important
relation which for rational polytopes is equivalent to this fact is
the following.

\begin{theorem}[Brianchon--Gram relation] If $\P$ is a polytope, then
  \[  
  \sum_{\F \subseteq \P} (-1)^{\dim \F} \omega_{\P}(\F) = 0\,,
  \]
  where the sum is over all faces $\F$ of $\P$.
\end{theorem}


\section{Formulas for the explicit computation of solid angles.} \label{formulasection}

In dimension 3, the following explicit formula can be used for
calculating solid angles.

\begin{proposition} \label{3dformula}
  Given a simplicial cone $K \subset \R^3$ at the origin, generated by the linearly independent vectors $v_1,
  v_2, v_3$, the solid angle $\omega_K$ at the origin is given by:
  \begin{align*}
    (4\pi) \omega_K &= \cos^{-1}\left( \frac{ (v_1\times v_2)\cdot (v_1
        \times v_3) }{\| v_1\times v_2 \| \| v_1 \times v_3 \|
      }\right) + \cos^{-1}\left( \frac{ (v_2\times v_1)\cdot (v_2
        \times v_3) }{\| v_2\times v_1 \| \| v_2 \times v_3 \| }
    \right) \\
    &\quad + \cos^{-1}\left( \frac{ (v_3\times v_1)\cdot (v_3 \times
        v_2) }{\| v_3\times v_1 \| \| v_3 \times v_2 \| }\right) -
    \pi,
  \end{align*}
  where $\times$ denotes the cross product of 3-dimensional vectors,
  $\cdot$ denotes the dot product, and $\| \|$ is the usual Euclidean norm.
\end{proposition}

\begin{proof} First note that computing $\omega_K$ is equivalent to
  taking a sphere of radius 1 at the origin and calculating the
  surface area of its intersection with $K$ divided by the surface
  area of the sphere, which is $4\pi$. The surface area of a spherical
  triangle is, as a consequence of Girard's theorem \cite[\S
  6.9]{coxeter}, the sum of its spherical angles minus $\pi$. The
  spherical angle at $v_i$ is precisely the dihedral angle $\theta_i$
  between the two faces of $K$ that intersect at $v_i$, and
  $\cos(\theta_i)$ is equal to the dot product of the normal vectors
  to the planes spanned by the faces, whence the formula.
\end{proof}

\begin{proof}[Proof of Proposition \ref{reeveprop}] 
  Given a simplex $\Delta$, let
  \[
  S(\Delta) := \sum_{{\bf v} \text{ a vertex}} \omega_{\Delta}({\bf
    v}) \, .
  \]
  Let $S = S(\Delta_h)$. By \eqref{macreciprocity}, the solid-angle
  polynomial of $\Delta_h$ is
  \[
  A_{\Delta_h}(t) = \frac{h}{6} t^3 + \left( S - \frac{h}{6} \right)
  t.
  \]
  Since $S < \frac{1}{2}$ by Proposition~\ref{upperbound} below, we
  conclude that $S - \frac{h}{6} < 0$ if $h \ge 3$. A direct
  calculation using Proposition~\ref{3dformula} shows that
  $S(\Delta_1) \approx 0.127 < \frac{1}{6}$ and $S(\Delta_2) \approx
  0.171 < \frac{1}{3}$.
\end{proof}

To handle the situation in any dimension $d$ we now describe a
formula, discovered by Aomoto \cite{solidangle} in 1977, that allows
us to use an infinite hypergeometric series to compute $\omega_K$ for
a simplicial $d$-dimensional cone $K$.  We follow Kenzi Sato's
exposition \cite{Sato}, as it clarifies Aomoto's fundamental work a
bit further.

We begin with the hyperplane description of a spherical simplex in
$\R^d$, defined by
\[
\Delta = \{ x \in \mathrm{S}^{d-1} \mid \langle n_i, x \rangle \geq 0,\ i = 0,
\dots, d -1 \},
\]
where the $n_i$ are linearly independent integer vectors, normal to
the facets of $\Delta$ (they are inward-pointing normal vectors).  We
define $\theta_{i,j}$ to be the dihedral angle between the two facets
whose normal vectors are $n_i$ and $n_j$.  Thus, we have
\[
\cos( \theta_{i,j}) = - \langle  n_i, n_j  \rangle.
\]

The solid angle $\omega_{\Delta}$, i.e. the volume of the spherical
simplex $\Delta$, is determined by the $d \choose 2$ dihedral angles
$\theta_{0,1}, \theta_{0,2}, \dots, \theta_{d-2, d-1}$.  Here is
Aomoto's hypergeometric series:
\[
\omega_\Delta = C \sum_{ \m \in \Z_{>0}^{ d(d-1)/2 } } \frac{\prod_{i
    < j} (-2 b_{i,j})^{m_{i,j}} }{ \prod_{i < j} m_{i,j} !  }
\prod_{k=0}^{d-1} \Gamma \left( \frac{1}{2}(m_{0,k} + \dots + m_{k-1,
    k} + m_{k, k+1} + \dots + m_{k, d-1}) + \frac{1}{2} \right),
\]
where the sum is extended over all integer vectors of the form $\m =
(m_{0,1}, m_{0,2}, \dots, m_{d-2, d-1}) \in \Z_{>0}^{d(d-1)/2}$, where
$\Gamma$ denotes the Euler gamma function, where 
$C = \sqrt{\det B}/\pi^{d/2}$, and where the matrix $B:= (b_{i,j})$ is
a Gram-like matrix defined as follows.  First, let
\[
G = \Big( \langle n_i, n_j \rangle \Big),
\]
a $d \times d$ Gram matrix.  Next, let 
\[
G_k = G \text { except with its $k$th row and $k$th column deleted,}
\]
so that $G_k$ is a $(d-1) \times (d-1)$ matrix. Next, let
\[
K_{i,j} = \delta(i,j) \frac{   \det G_i}{   \det G},
\]
a diagonal matrix, by definition of the Kronecker delta function
$\delta(i,j)$.  Finally, let
\[
B = (b_{i,j}) = K^{-1} G^{-1} K^{-1}.
\]

In a recent paper, Ribando \cite{ribando} rediscovered Aomoto's
results, but gave different proofs, so that his paper has the
redeeming feature of having a somewhat simplified proof of a simpler
version of Aomoto's results.


\section{Solid angles at the vertices of a polytope} \label{vertexsection}

As before, given any convex polytope $\P$, let $S(\P)$ denote the sum
of the solid angles at the vertices of $\P$. Our goal in this section
is to study the extremal behavior of $S(\P)$, and especially in the
case that $\P$ is a simplex.  In passing, we note that a conjecture of
\cite[Chapter 12]{ccd} that the regular simplex minimizes $S(\Delta)$
is false, but that similar questions on minimizing or maximizing
$S(\P)$ are still quite interesting.  In particular, \cite{Barnette}
David Barnette has given an amusing and beautiful equivalence for the
minimization of $S(\P)$ in terms of the existence of a Hamiltonian
circuit along the edge graph of $\P$.  There are also two papers by
Perles and Shephard (\cite{PerlesShepard1} and \cite{PerlesShepard2})
with very general results along these lines as well. The combinatorial
type of a polytope shall refer to isomorphism type of its face
lattice.

\begin{theorem}[Barnette] \label{MinimumSolidAngleSum} Any
  $3$-dimensional polytope $\P$ has a Hamiltonian circuit along its
  edge-graph if and only if there are polytopes with the same
  combinatorial type as $\P$, with arbitrarily small vertex angle
  sums.
 \end{theorem}

 Since it is obvious that a simplex has a Hamiltonian circuit along
 its edge graph, there are always simplices that have arbitrarily
 small vertex angle sums.  It is also worth noting that
 \cite{Barnette} has a general upper bound for $S(\P)$ in any
 dimension.

 These results about $S(\P)$ are mostly existential, so we complement
 them with some constructive examples below.  We now compute some
 explicit examples of solid-angle polynomials that we shall need
 later.

\begin{proposition} \label{permutationsimplex} Let $\pi \in
  \mathfrak{S}_d$ be a permutation on $d$ elements, and let $\e_1,
  \dots, \e_d$ be the standard basis vectors. The polytope $\Delta_\pi
  = \conv\{\e_{\pi(1)}, \e_{\pi(1)} + \e_{\pi(2)}, \dots, \e_{\pi(1)} +
  \cdots + \e_{\pi(d)}\}$ has solid angle polynomial $\frac{1}{d!}t^d$.
\end{proposition}

\begin{proof} The set $\{ \Delta_\pi \mid \pi \in \mathfrak{S}_d \}$
  is a triangulation of the unit cube $[0,1]^d$, and the $\Delta_\pi$
  are all congruent to one another, i.e., any such simplex can be
  obtained from any other through a series of rotations, reflections,
  and translations. Note that $[0,t]^d$ is tiled by $t^d$ copies of
  $[0,1]^d$, so $A_{[0,1]^d}(t) = t^d$ since $A_{[0,1]^d}(1) = 1$.
  Hence $A_{\Delta_\pi}(t) = \frac{1}{d!}t^d$.
\end{proof}

Now let $\Delta = \conv\{(0,0,0), (0,1,1), (1,0,1), (1,1,0)\}$, which
is a regular tetrahedron. The solid angle at an edge is the dihedral
angle $\frac{1}{2\pi} \cos^{-1}(\frac{1}{3})$, and they are all the
same by symmetry, so Brianchon--Gram gives
\[
0 = -1 + 4 \cdot \frac{1}{2} -
\frac{3}{\pi}\cos^{-1}\left(\frac{1}{3}\right) + 4\omega\,,
\]
where $\omega$ is the solid angle at a vertex. Thus $S(\Delta) \approx
0.175$. But by Proposition \ref{permutationsimplex}, the solid-angle
polynomial of $\mathcal{Q} = \conv\{(0,0,0), (0,0,1), (0,1,1),
(1,1,1)\}$ is $\frac{1}{6}t^3$, so $S(\mathcal{Q}) = \frac{1}{6}$,
which is less than 0.175. Also, the sum of the solid angles at the
vertices of the standard simplex $\conv\{(0,0,0), (1,0,0), (0,1,0),
(0,0,1)\}$ is approximately 0.206, so $S(\Delta)$ is neither a maximum
nor a minimum. Using Proposition~\ref{solidanglelimits}, one observes
the same behavior in higher dimensions.

Despite our negative answer for the original conjecture, we can
rephrase it as follows: In fixed dimension, which simplices
minimize/maximize $S(\Delta)$? We note that a similar question of
angle sums is addressed in \cite{anglesums}, but here we are concerned
with integral polytopes.

\begin{proposition} \label{solidanglelimits} For $d>2$, let $\e_1, \e_2,
  \dots, \e_d$ be the standard basis vectors of $\R^d$, and define
  $\Delta(h_1, \dots, h_{d-1}) = \conv \{{\bf 0}, \e_1, \e_2, \dots,
  \e_{d-1}, (h_1, h_2, \dots, h_{d-1}, 1)\}$. Then 
  \begin{enumerate}[\rm (a)]
  \item \label{negativecase} $S(\Delta(h_1, \dots, h_{d-1}))$ is
    arbitrarily close to $\frac{1}{2}$ for sufficiently large negative
    values of all of the $h_i$'s, and
  \item \label{positivecase} $S(\Delta(h, h, 1, 1, \dots, 1)) \to 0$
    as $h \to +\infty$.
  \end{enumerate}
\end{proposition}

\begin{proof}
  We first prove (\ref{negativecase}). To show that $S(\Delta(h_1,
  \dots, h_{d-1})) \to \frac{1}{2}$ as $h_i \to -\infty$, it is enough
  to show that the solid angle at the origin approaches $\frac{1}{2}$
  by Proposition \ref{upperbound} below. Fix $\eps > 0$ sufficiently
  small. (It is enough to choose $\eps$ such that any point $x$ within
  $\eps$ of the origin satisfies $|x_1| + \cdots + |x_d| < 1$.)  To
  show that the solid angle at the origin approaches $\frac{1}{2}$, it
  is enough to show that any point with positive $x_d$ coordinate that
  is within $\eps$ of the origin is contained in $\Delta(h_1, \dots,
  h_{d-1})$ for sufficiently negative values of $h_i$. But this is
  clear: given such a point $\x = (x_1, \dots, x_d)$, assume for
  notational simplicity that its first $k$ coordinates and $x_d$ are
  the only positive entries. Then
  \[
  \x = (1-x_1-\cdots-x_k-x_d){\bf 0} + x_1\e_1 + \cdots + x_k\e_k +
  x_d(h_1, \dots, h_{d-1}, 1)
  \]
  with $h_1 = \cdots = h_k = 0$ and $h_i = x_i/x_d$ for $i=k+1, \dots,
  d-1$. In the case that $\x$ has different positive entries other
  than the first $k$ coordinates, it is trivial to modify the above
  argument.

  For (\ref{positivecase}), we need to show that the solid angle at
  each vertex approaches 0 as $h$ tends to $+\infty$. We do so by
  induction on dimension, starting with $d=3$ where we can use
  Proposition~\ref{3dformula}. Then the solid angles at ${\bf 0},
  \e_1, \e_2, (h,h,1)$, respectively, are as follows:
  \begin{align*}
    &\frac{1}{4\pi} \left( 2\cos^{-1}\left(
        \frac{h}{\sqrt{h^2+1}}\right) + \cos^{-1}\left(
        \frac{-h^2}{h^2+1}\right) - \pi\right),\\
    &\frac{1}{4\pi} \left( \cos^{-1}\left( \frac{h}{\sqrt{h^2+1}}
      \right) + \cos^{-1}\left( \frac{-2h+1}{\sqrt{4h^2-4h+3}}\right)
      + \cos^{-1}\left(
        \frac{2h^2-h+1}{\sqrt{(h^2+1)(4h^2-4h+3)}}\right) -
      \pi\right),\\
    &\frac{1}{4\pi} \left( \cos^{-1}\left( \frac{h}{\sqrt{h^2+1}}
      \right) + \cos^{-1}\left( \frac{-2h+1}{\sqrt{4h^2-4h+3}}\right)
      + \cos^{-1}\left(
        \frac{2h^2-h+1}{\sqrt{(h^2+1)(4h^2-4h+3)}}\right) -
      \pi\right),\\
    &\frac{1}{4\pi} \left( \cos^{-1}\left(
        \frac{-h^2}{\sqrt{h^2+1}}\right) + 2\cos^{-1}\left(
        \frac{2h^2-h+1}{\sqrt{(h^2+1)(4h^2-4h+3)}}\right) - \pi\right),
  \end{align*}
  and these all tend to 0 as $h \to \infty$.

  Now let $\Delta_d$ be the $d$-dimensional simplex
  $\Delta(h,h,1,\dots,1)$. Then $\Delta_{d+1} \subset \Delta_d \times
  [0,1]$, and the vertices of $\Delta_{d+1}$ are a subset of the
  vertices of $\Delta_d \times [0,1]$. Hence $S(\Delta_{d+1}) \le
  S(\Delta_d \times [0,1])$, and we finish by using
  Lemma~\ref{inductionlemma} below.
\end{proof}

\begin{lemma} \label{inductionlemma} Given a polyhedral cone $K$ with
  vertex $\x$, one has 
  \[
  \omega_{K \times [0,1]}((\x,0)) = \omega_{K \times [0,1]}((\x, 1)) \le
  c \, \omega_K(\x)
  \]
  where $c$ is a constant that only depends on $d = \dim K$.
\end{lemma}

\begin{proof} The first equality follows by symmetry. We can write
  \[
  \omega_K(\x) = \frac{\vol(B(1,\x) \cap K)}{\vol(B(1,\x))}
  \]
  since $K$ is a polyhedral cone. Now let $C = B(1,\x) \times
  [0,1]$. The upper hemisphere of the $(d+1)$-dimensional ball of
  radius 1 centered at $(\x,0)$ is contained in $C$. Hence we have
  \[
  \vol(B(1,\x) \cap K) = \vol(C \cap (K \times [0,1])) \ge \vol(B(1,
  (\x,0)) \cap (K \times [0,1])),
  \]
  and setting 
  \[
  c = \frac{\vol(B(1,\x))}{\vol(B(1,(\x,0)))}
  \]
  proves the inequality. 
\end{proof}



Notice that 0 is an obvious lower bound for $S(\Delta)$. It turns out
that $\frac{1}{2}$ is the upper bound, as was already shown in
\cite{Barnette}.  We provide another proof here, with a pretty
argument due to Herbert Edelsbrunner and Igor Rivin (personal
communication).

\begin{proposition} \label{upperbound} Let $\Delta$ be a
  $d$-simplex. For $d = 2$, $S(\Delta) = \frac{1}{2}$, and for $d>2$,
  $S(\Delta) < \frac{1}{2}$.  
\end{proposition}

\begin{proof} We may assume that one of the vertices is the
  origin. Let $\v_1, \v_2, \dots, \v_d$ be the other vertices of
  $\Delta$. Let $K$ be the cone generated by $\v_1, \dots, \v_d$. The
  fundamental parallelepiped defined by 
  \[
  \Pi = \{ \lambda_1\v_1 + \cdots + \lambda_d\v_d \mid 0 \le \lambda_i
  < 1\}
  \]
  tiles $K$, and thus the sum of the solid angles of the vertices of
  its closure is 1. Define
  \begin{align*}
    \Delta_1 &= \{\lambda_1\v_1 + \cdots + \lambda_d\v_d \mid 0 \le
    \lambda_i,\ \sum \lambda_i \le 1 \},\\
    \Delta_2 &= \{(1-\lambda_1)\v_1 + \cdots + (1-\lambda_d)\v_d \mid
    0 \le \lambda_i,\ \sum \lambda_i \le 1 \},
  \end{align*}
  which are subsets of $\Pi$. Then $\Delta_1 \cap \Delta_2 =
  \varnothing$ and $\Delta_1 \cup \Delta_2 \ne \Pi$ if $d>2$. So $\Pi$
  contains two congruent disjoint copies of $\Delta$, and hence
  $2S(\Delta) < 1$ when $d>2$.
\end{proof}


\section{Numerator polynomial of solid-angle
  series.} \label{nonnegativesection}

\begin{proof}[Proof of Theorem \ref{nonnegativethm}]
  Since solid angles are additive, it suffices to prove the statement
  for lattice simplices. If $\Delta = \conv \left\{ \v_1, \v_2, \dots,
    \v_{ d+1 } \right\} \subset \R^d$ is a lattice $d$-simplex, we
  form the \emph{cone over} $\Delta$:
  \begin{equation}\label{coneoverdef}
    \cone(\Delta) := \left\{ \lambda_1 (\v_1, 1) + \lambda_2 (\v_2, 1) +
      \cdots + \lambda_{ d+1 } (\v_{ d+1 }, 1) \mid \lambda_1,
      \lambda_2, \dots, \lambda_{ d+1 } \ge 0 \right\} \subset
    \R^{d+1}.
  \end{equation}
  We now consider codimension-1 solid angles in $\R^{ d+1 }$, by
  setting $f_{\cone(\Delta)}(\x)$ of a point $\x \in \cone(\Delta)$ to
  be the solid angle of $\x$ relative to the hyperplane through $\x$
  with normal vector $(0, 0, \dots, 0, 1)$. To be more precise, let $H
  = \left\{ \x \in \R^{ d+1 } \mid x_{ d+1 } = 0 \right\}$, then
  \[
  f_{\cone(\Delta)}(\x) = \omega_{ \cone(\Delta) \cap (x_{ d+1 } + H)}
  (\x) \,,
  \]
  where we are treating $x_{d+1} + H$ as an isomorphic copy of $\R^d$.
  Now we need a generating function that lists all function values of
  $f$ of the lattice points in a polyhedron $\P' \subset \R^{ d+1 }$:
  \[
  g_{\P'}(\z) := \sum_{\m \in \Z^{d+1}} f_{\P'}(\m) \, \z^\m .
  \]
  Here we are using the multivariate notation $\z^\m = z_1^{ m_1 }
  z_2^{ m_2 } \cdots z_{ d+1 }^{ m_{ d+1 } }$. As a function of $\P'$,
  $g_{\P'}(\z)$ is totally additive as long as the involved polyhedra
  have no facets parallel to $H$. We have now set the stage to use the
  machinery of \cite[Chapter 3]{ccd}, which in fact goes back to
  Ehrhart's original ideas.  The cone over $\Delta$ can be tiled with
  translates of the parallelepiped
  \begin{equation}\label{tilingeq}
    \Pi := \left\{ \lambda_1 (\v_1, 1) + \lambda_2 (\v_2, 1) + \cdots +
      \lambda_{ d+1 } (\v_{ d+1 }, 1) \mid 0 \le \lambda_1, \lambda_2,
      \dots, \lambda_{ d+1 } < 1 \right\} ,
  \end{equation}
  by nonnegative integer combinations $\sum n_i\v_i$. The total
  additivity of $g_{\P'}(\z)$ implies that
  \[
  g_{ \cone(\Delta) } (\z) = \frac{ g_{ \Pi } (\z) }{ \left( 1 - \z^{
        (\v_1, 1) } \right) \left( 1 - \z^{ (\v_2, 1) } \right) \cdots
    \left( 1 - \z^{ (\v_{d+1}, 1) } \right) } \, .
  \]
  Setting all but the last variable equal to 1 gives the generating
  function of the solid-angle polynomial of $\Delta$:
  \[
  \Solid_\Delta (z) = g_{ \cone(\Delta) } (1,1, \dots, 1, z) = \frac{
    g_\Pi (1, 1, \dots, 1, z) }{ (1-z)^{ d+1 } } \, .
  \]
  The coefficient of $z^k$ in the polynomial $g_\Pi (1, 1, \dots, 1,
  z)$ records the solid-angle sum of the points in $\Pi \cap \left\{
    \x \in \Z^{ d+1 } \mid x_{ d+1 } = k \right\}$, which is positive
  for $1 \le k \le d$ and 0 for $k=0$.
\end{proof}

It is tempting to conjecture that the coefficients of the numerator
polynomial form a unimodal sequence because of the palindromy, but
this turns out not to be the case.  For example, let $\Delta$ be a lattice 3-simplex
whose only integer points are its vertices, and let $S$ be the sum of
the solid angles at the vertices of $\Delta$. Then 
\[
A_\Delta(t) = \frac{1}{6} t^3 + \left(S-\frac{1}{6}\right) t\,,
\]
so
\[
\Solid_\Delta(z) = \frac{Sz^3 + (1-2S)z^2 + Sz}{(1-z)^4}\,.
\]
If the numerator polynomial is unimodal, then $1-2S \ge S$, which
implies $S \le \frac{1}{3}$. In Section~\ref{vertexsection}, we gave a
class of simplices for all dimensions whose only integer points are
its vertices and whose solid-angle sum $S$ converges to
$\frac{1}{2}$. A similar computation in dimension 4 shows that $S >
\frac{1}{4}$ means that $\Delta$ does not have a unimodal numerator
polynomial.

It would be interesting to find nice conditions for when the numerator
polynomial is unimodal, however. In dimension 3, if $\vol(\P) t^3 +
ct$ is the solid-angle polynomial of a polytope $\P$, then the
numerator polynomial is 
\[
(\vol(\P)+c)z^3 + (4\vol(\P) - 2c) z^2 + (\vol(\P) + c)z,
\]
so is unimodal if and only if $c \le
\vol(\P)$. Proposition~\ref{reeveprop} gives an infinite family of
3-polytopes whose solid-angle polynomial has $c < 0$, which at least
gives some examples.


\section{Nonnegativity results for valuation generating functions.} \label{sec:valuation}

Let $\nu$ be a valuation and write 
\[
G_{\P}(z) = \frac{a_d(\P) z^d + a_{d-1}(\P) z^{d-1} + \cdots + a_0(\P)}{(1-z)^{d+1}}
\]
for a lattice polytope $\P \subset \R^d$. In the previous version of the article, we stated that if $\nu$ is translation-invariant and nonnegative, then $a_j(\P) \ge 0$ for $j = 0, 1,  \dots, d$, and furthermore, if $\P \subseteq \mathcal{Q}$ where $\mathcal{Q}$ is also a lattice polytope, then $a_j(\P) \le a_j(\mathcal{Q})$ for $j=0,1,\dots,d$.

However, the proofs for these results contained an error, and indeed, they are not correct as stated. This was pointed out to us by Katharina Jochemko and Raman Sanyal. The main result of their paper \cite{jochemko-sanyal} gives a description of those valuations $\nu$ for which the above two statements are correct.


\section{Some remarks on period collapse.} \label{periodsection}

Suppose $\P$ is a rational polytope. Then $A_\P(t)$ is a
quasipolynomial, whose period divides the least common multiple of the
denominators of the vertex coordinates of $\P$. For a generic
polytope, this number \emph{is} the period of $A_\P(t)$. However, in
analogy with Ehrhart quasipolynomials, we expect that there are
special classes of polytopes that exhibit \emph{period collapse},
i.e., $A_\P(t)$ is a polynomial. Here is one such family:

\begin{proposition} The polytope $\P = [0,\frac{1}{2}] \times
  [0,1]^{d-1}$ has solid-angle polynomial $\frac{1}{2}t^d$.
\end{proposition}

\begin{proof} 
  As seen in the proof of Proposition~\ref{permutationsimplex}, the
  solid-angle polynomial of the unit cube is $t^d$. In fact, $\P$ and
  $\P' = [\frac{1}{2},1] \times [0,1]^{d-1}$ are congruent to one
  another, so have the same solid-angle enumerator. Since their union
  is the unit cube, we conclude the desired result by additivity.
\end{proof}

We can completely classify period collapse in dimension 1. Let $\P =
[a,b]$. There are two cases to consider. In the first case, $a \in \Z$
and $b \in \Q \setminus \Z$. Let $f_0, \dots, f_{p-1}$ be the
constituent polynomials of $A_{\P}(t)$; that is, $A_\P(t) = f_j(t)$ if
$t \equiv j \bmod p$.  Then $f_0 = (b-a)t$, so in order for
$A_{\P}(t)$ to be a polynomial, we need that $A_{\P}(1) = b-a$. We
know that $A_{\P}(1) = \lfloor b\rfloor - a + \frac{1}{2}$, so we
conclude $b = \lfloor b\rfloor + \frac{1}{2}$, and thus $\P$ must have
denominator 2. If both $a$ and $b$ are nonintegral, then $A_{\P}(1) =
\lfloor b\rfloor - \lceil a\rceil + 1$, so for period collapse to
occur, this must be equal to $b-a$. This means that $\P$ is of the
form $[\frac{p}{q}, \frac{p}{q} + n]$ where $n \in \Z$. For $t \nmid
q$, there is a bijection between points of $t\P$ and the points of
$(0,tn]$, so $A_{\P}(t) = nt$ is given by adding $\frac{p}{q}$ to
$(0,tn]$. We summarize these results.

\begin{proposition} Let $\P = [a,b]$ be a 1-dimensional rational
  polytope. Then $A_{\P}(t)$ is a polynomial if and only if one of the
  following holds:
  \begin{enumerate}[{\rm (1)}]
  \item $a \in \Z, 2b \in \Z$ or $2a \in \Z, b \in \Z$
  \item $b-a \in \Z$
  \end{enumerate}
\end{proposition}

In higher dimensions we can at least say the following. For $0 \le j
\le d = \dim \P$, define the {\it $j$-index} of $\P$ to be the minimal
positive integer $p_j$ such that the affine span of each
$j$-dimensional face of $\P$ contains an integer point. Then the $p_j$
give bounds on the periods of coefficient functions of $A_\P(t)$.

\begin{proposition} Given a rational $d$-polytope $\P$ with $j$-index
  $p_j$, write the solid-angle polynomial as $A_\P(t) = c_d(t) t^d +
  c_{d-2}(t) t^{d-2} + \cdots + c_1(t) t + c_0(t)$ (where $c_1(t) = 0$
  if $d$ is even, and $c_0(t) = 0$ if $d$ is odd). Then the minimum
  period of $c_j(t)$ divides $p_j$ for all $j$.
\end{proposition}

\begin{proof} The $j$-index of a face of $\P$ divides $p_j$, so this
  follows from the identity \eqref{polysum} and the fact that the
  corresponding theorem for Ehrhart polynomials holds (see
  \cite[Theorem 6]{mcmullenreciprocity}).
\end{proof}


\def\cprime{$'$} \def\cprime{$'$}
\providecommand{\MR}{\relax\ifhmode\unskip\space\fi MR }
\providecommand{\MRhref}[2]{%
  \href{http://www.ams.org/mathscinet-getitem?mr=#1}{#2}
}
\providecommand{\href}[2]{#2}

\end{document}